\documentclass[12pt]{amsart}
\usepackage{amsfonts,amssymb,amscd,amsmath,enumerate,verbatim}
\usepackage[latin1]{inputenc}
\usepackage{amscd}
\usepackage{latexsym}
\usepackage{pstcol,pst-plot,pst-3d}
\usepackage{mathptmx}
\usepackage{multicol}

\psset{unit=0.7cm,linewidth=0.8pt,arrowsize=2.5pt 4}

\newpsstyle{fatline}{linewidth=1.5pt}
\newpsstyle{fyp}{fillstyle=solid,fillcolor=verylight}
\definecolor{verylight}{gray}{0.97}
\definecolor{light}{gray}{0.9}
\definecolor{medium}{gray}{0.85}

\def\frk{\mathfrak}               

\def\Phi{{\frk N}}
\def\opn#1#2{\def#1{\operatorname{#2}}} 
\opn\chara{char} \opn\length{\ell} \opn\pd{pd} \opn\rk{rk}
\opn\projdim{proj\,dim} \opn\injdim{inj\,dim} \opn\rank{rank}
\opn\depth{depth} \opn\grade{grade} \opn\height{height}
\opn\embdim{emb\,dim} \opn\codim{codim}

\opn\Tr{Tr} \opn\bigrank{big\,rank}
\opn\superheight{superheight}\opn\lcm{lcm}
\opn\trdeg{tr\,deg}
\opn\reg{reg} \opn\lreg{lreg} \opn\ini{in} \opn\lpd{lpd}
\opn\size{size}\opn{\mult}{mult}
\opn\div{div} \opn\Div{Div} \opn\cl{cl} \opn\Cl{Cl}
\opn\Spec{Spec} \opn\Supp{Supp} \opn\supp{supp} \opn\Sing{Sing}
\opn\Ass{Ass} \opn\Min{Min}
\opn\Ann{Ann} \opn\Rad{Rad} \opn\Soc{Soc}
\opn\Syz{Syz} \opn\Im{Im} \opn\Ker{Ker} \opn\Coker{Coker}
\opn\Am{Am} \opn\Hom{Hom} \opn\Tor{Tor} \opn\Ext{Ext}
\opn\End{End} \opn\Aut{Aut} \opn\id{id} \opn\ini{in}

\opn\nat{nat}
\opn\pff{pf}
\opn\Pf{Pf} \opn\GL{GL} \opn\SL{SL} \opn\mod{mod} \opn\ord{ord}
\opn\Gin{Gin}
\opn\Hilb{Hilb}\opn\adeg{adeg}\opn\std{std}\opn\ip{infpt}
\opn\Pol{Pol}
\opn\sat{sat}
\opn\Var{Var}
\opn\Gen{Gen}
\opn\indmatch{indmatch}

\opn\aff{aff} \opn\con{conv} \opn\relint{relint} \opn\st{st}
\opn\lk{lk} \opn\cn{cn} \opn\core{core} \opn\vol{vol}
\opn\link{link} \opn\star{star}
\opn\gr{gr}

\def\Cc{{\mathcal C}}
\def\pot#1#2{#1[\kern-0.28ex[#2]\kern-0.28ex]}

\opn\dirlim{\underrightarrow{\lim}}
\opn\inivlim{\underleftarrow{\lim}}
\def\Implies{\ifmmode\Longrightarrow \else
        \unskip${}\Longrightarrow{}$\ignorespaces\fi}
\def\implies{\ifmmode\Rightarrow \else
        \unskip${}\Rightarrow{}$\ignorespaces\fi}
\def\iff{\ifmmode\Longleftrightarrow \else
        \unskip${}\Longleftrightarrow{}$\ignorespaces\fi}

\let\:=\colon
\newtheorem{Theorem}{Theorem}[section]
\newtheorem{Lemma}[Theorem]{Lemma}

\newtheorem{Definition}[Theorem]{Definition}

\newtheorem{Conjecture}[Theorem]{Conjecture}

\let\epsilon\varepsilon
\let\phi=\varphi
\let\kappa=\varkappa
\def\qed{\ifhmode\textqed\fi
      \ifmmode\ifinner\quad\qedsymbol\else\dispqed\fi\fi}
\def\textqed{\unskip\nobreak\penalty50
       \hskip2em\hbox{}\nobreak\hfil\qedsymbol
       \parfillskip=0pt \finalhyphendemerits=0}
\def\dispqed{\rlap{\qquad\qedsymbol}}

\opn\dis{dis}
\def\pnt{{\raise0.5mm\hbox{\large\bf.}}}

\opn\Lex{Lex}

\newcommand{\inD}[1][\relax]{\def\argone{#1}\def\temprelax{\relax}
  \ifx\argone\temprelax\right.\else\,\middle|#1\right.{}\fi}

\newif\ifbinary
\binarytrue

\begin{document}

\title{ On the binomial edge ideals of block graphs}

\author{Faryal Chaudhry, Ahmet Dokuyucu, Rida Irfan}

\address{Abdus Salam School of Mathematical Sciences, GC University,
Lahore. 68-B, New Muslim Town, Lahore 54600, Pakistan} \email{chaudhryfaryal@gmail.com}

\address{Faculty of Mathematics and Computer Science, Ovidius University
Bd. Mamaia 124, 900527 Constanta \\
and Lumina-The University of South-East Europe
Sos. Colentina nr. 64b, Bucharest,
Romania}
\email{ahmet.dokuyucu@lumina.org}

\address{Abdus Salam School of Mathematical Sciences, GC University,
Lahore. 68-B, New Muslim Town, Lahore 54600, Pakistan} \email{ridairfan\_88@yahoo.com}

\thanks{The  first and the third author  were supported by the Higher Education Commission of Pakistan and the Abdus Salam School of Mathematical Sciences, Lahore, Pakistan.}

\begin{abstract}
We find a class of block graphs whose binomial edge ideals have minimal regularity. As a consequence,  we characterize the trees whose binomial edge ideals have minimal regularity. Also, we show that the binomial edge ideal of a block graph has the same depth as its initial ideal.
\end{abstract}


\maketitle

\section{Introduction}

In this paper we study homological properties of some classes of binomial edge ideals.

Let $G$ be a simple graph on the vertex set $[n]$ and let $S = K[x_1,\ldots,x_n,y_1,\ldots,y_n]$ be the polynomial ring in $2n$ variables over a field $K$. For $  1 \leq {i}<{j} \leq n$, we set $f_{ij} = x_{i}y_{j}-x_{j}y_{i}$. The binomial edge ideal of $G$ is defined as $J_{G} = (f_{ij}: \{i,j\} \in E(G)).$
Binomial edge ideals were introduced in \cite{HHHKR} and  \cite{Oh}. Algebraic and homological properties of binomial edge ideals have been studied in several papers.
In \cite{EHH}, it was conjectured that $J_{G}$ and $\ini_{<}(J_{G})$ have the same extremal Betti numbers. Here $<$ denotes the lexicographic order in $S$ induced by $x_1>x_2> \cdots >x_n>y_1>y_2>\cdots>y_n$.  This conjecture was proved in \cite{AD} for cycles and complete bipartite graphs.
In \cite{EZ}, it was shown that, for a closed graph $G$, $J_{G}$ and $\ini_{<}(J_{G})$ have the same regularity which can be expressed in the combinatorial data of the graph.
We recall that a graph $G$ is closed if and only if it has a quadratic Gr\"obner basis with respect to the lexicographic order.

In support of the conjecture given in \cite{EHH}, we show, in Section~\ref{block}, that if $G$ is a block graph, then $\depth(S/J_{G}) = \depth(S/\ini_{<}(J_{G}))$; see Theorem~\ref{2.2}. By a block graph we mean a chordal graph $G$ with the property that any two maximal cliques of $G$ intersect in at most one vertex.

Also, in the same section, we show a similar equality for regularity. More precisely, in Theorem~\ref{clgraph} we show that $\reg(S/J_G)=\reg(S/\ini_< (J_G))=\ell$ if $G$ a $C_{\ell}$-graph. $C_{\ell}$-graphs constitute a subclass of the block graphs; see Section~\ref{block} for definition and Figure~\ref{example} for an example.

In \cite{M} it was shown that, for any connected graph $G$ on the vertex set $[n]$, we have
\[\ell \leq \reg(S/J_{G}) \leq n-1,\] where $\ell$ is the length of the longest induced path of $G$.

The main motivation of our work was to answer the following question.
May we characterize the connected graphs $G$ whose longest induced path has length $\ell$ and $\reg(S/J_G)=\ell ?$
We succeeded to answer this question for trees. In Theorem~\ref{3.1}, we show that if $T$ is a tree whose longest induced path has length $\ell$, then $\reg(S/J_{T})=\ell$ if and only if $T$ is caterpillar. 
A caterpillar tree is a tree $T$ with the property that it contains a path $P$ such that any vertex of $T$ is either a vertex of $P$ or it is adjacent to a vertex of $P$.

In \cite{Oh}, the so-called weakly closed graphs were introduced. This is a class of graphs which includes closed graphs. In the same paper, it was shown that a tree is caterpillar if and only if it is a weakly closed graph.
Having in mind our Theorem~\ref{3.1} and Theorem 3.2 in \cite{EZ} which states that $\reg(S/J_{G})=\ell$ if $G$ is a connected closed graph whose longest induced path has length $\ell$, and by some computer experiments, we are tempted to formulate the following.

\begin{Conjecture}
If $G$ is a connected weakly closed graph whose longest induced path has length $\ell$, then $\reg(S/J_{G})=\ell$.
\end{Conjecture}

\section{Preliminaries}
\label{preliminaries}
In this section we introduce the notation used in this paper and summarize a few results on  binomial edge ideals.

Let $G$ be a simple graph on the vertex set $[n] = \{1,\ldots,n\} $, that is, $G$ has no loops and no multiple edges. Furthermore, let $K$ be a field and $S = K[x_1,\ldots,x_n,y_1,\ldots,y_n]$ be the polynomial ring in $2n$ variables. For $  1 \leq {i}<{j} \leq n$, we set $f_{ij} = x_{i}y_{j}-x_{j}y_{i}$. The \textit{binomial edge ideal} $J_{G}$ $\subset S$ associated with ${G}$ is generated by all the quadratic binomials $f_{ij}= x_{i}y_{j}-x_{j}y_{i}$ such that $\left\{i,j\right\} \in E(G)$. Binomial edge ideals were introduced in the papers \cite{HHHKR} and \cite{Oh}.

We first recall some basic definitions from graph theory. A vertex $i$ of $G$ whose deletion from the graph gives a graph with more connected components than $G$ is called a {\em cut point} of G. A {\em chordal} graph is a graph without cycles of length greater than or equal to $4.$ A {\em clique} of a graph $G$ is a complete subgraph of $G.$ The cliques of a  graph $G$ form a simplicial complex, $\Delta(G),$ which is called the {\em clique complex } of $G.$ Its facets are the maximal cliques of $G.$ A graph $G$ is a {\em block graph} if and only if it is  chordal and every two maximal cliques have at most one vertex in common. This class was considered in \cite[Theorem 1.1]{EHH}.

The clique complex $\Delta(G)$ of a chordal graph $G$ has the property that there exists a {\em leaf order} on its facets.  This means that the facets of $\Delta(G)$ may be ordered as $F_1,\ldots, F_r$ such that, for every $i>1,$ $F_i$ is a leaf of 
the simplicial complex generated by $F_1,\ldots,F_i$. A {\em leaf} $F$ of a simplicial complex $\Delta$ is a facet of $\Delta$ with the property that there exists another facet of $\Delta$, say $G$, such that, for every facet $H\neq F$ of $\Delta$, $H\cap F\subseteq G\cap F.$

\subsection{Gr\"obner bases of binomial edge ideals}
Let $<$ be the lexicographic order on $S$ induced by the natural order of the variables. As it was shown in \cite{HHHKR}, the Gr\"obner basis of $J_{G}$ with respect to this order may be given in terms of the admissible paths of $G$.

\begin{Definition}\cite{HHHKR}{\em \
Let $i<j$ be two vertices of $G$. A path $i=i_{0},i_{1},\ldots,i_{r-1},i_{r}=j$ from $i$ to $j$ is called  {\em admissible} if the following conditions are fulfilled:
\begin{enumerate}
  \item $i_{k}\neq i_{l}$ for $k\neq l$;
  \item for each $k=1,\ldots,r-1$ on has either $i_{k}<i$ or $i_{k}>j$;
  \item for any proper subset $\{j_{1},\ldots,j_{s}\}$ of $\{i_{1},\ldots,i_{r-1}\}$, the sequence $i,j_{1},\ldots,j_{s},j$ is not a path in $G$.
\end{enumerate}
Given an admissible path $\pi$ in $G$ from $i$ to $j$, we set $u_{\pi}= (\prod_{i_{k}>j} x_{i_{k}})(\prod_{i_{l}<i} y_{i_{l}})$.}
\end{Definition}

\begin{Theorem}\cite{HHHKR}\label{*}
The set of binomials
\[ \Gamma=\bigcup_{i<j}\{u_{\pi} f_{ij}~:~\pi~is ~an~ admissible~path~from~i~to~j\}\]
is the reduced Gr\"obner basis of $J_{G}$ with respect to the lexicographic order on $S$ induced by $x_{1}>\cdots>x_{n}>y_{1}>\cdots>y_{n}$.

\end{Theorem}
\subsection{Primary decomposition of binomial edge ideals}
Theorem~\ref{*} shows, in particular, that $\ini_{<}(J_{G})$ is a radical monomial ideal which implies that the binomial edge ideal $J_{G}$ is radical as well. Hence $J_{G}$ is equal to the intersection of all its minimal primes. These minimal prime ideals were determined in \cite{HHHKR}.

Let $\mathcal{S}\subset [n]$ be a (possible empty) subset of $[n]$ and let $G_{1},\ldots,G_{c(\mathcal{S})}$ be the connected components of $G_{[n]\setminus \mathcal{S}}$, where $G_{[n]\setminus \mathcal{S}}$ denotes the restriction of $G$ to the vertex set $[n]\setminus \mathcal{S}$. For $1\leq i \leq c(\mathcal{S})$, let $\widetilde{G_{i}}$ be the complete graph on the vertex set of $G_{i}$. Let

\[
P_{\mathcal{S}}(G)=(\bigcup_{{i}\in \mathcal{S}} \{x_{i},y_{i}\},J_{\tilde{G_{1}}},\ldots,J_{\tilde{G}_{c(\mathcal{S})}}).
\]

\begin{Theorem}\cite{HHHKR}
\[J_{G}=\bigcap_{{\mathcal{S}\subset[n]}}P_{\mathcal{S}}(G)
\]

 In particular, the minimal primes of $J_{G}$ are among $P_{\mathcal{S}}(G)$, where $\mathcal{S} \subset [n]$.
\end{Theorem}

\begin{Theorem}\cite{HHHKR}
Let G be a connected graph on the vertex set [n], and $\mathcal{S} \subset [n]$. Then $P_{\mathcal{S}}(G$) is a minimal prime ideal of $J_{G}$ if and only if $\mathcal{S} = \emptyset $ or $\mathcal{S} \neq \emptyset$ and for each ${i} \in \mathcal{S}$ one has $c(\mathcal{S} \setminus \{i\}) < c(\mathcal{S})$.
\end{Theorem}
A set $\mathcal{S}$ which satisfies the condition in the above theorem is called a \textit{cut-point} set of $G$.

\section{Initial ideals of binomial edge ideals of block graphs}
\label{block}

In this section, we first show that, for a block graph $G$ on $[n]$ with $c$ connected components, we have $\depth(S/J_{G})=\depth(S/\ini_{<}(J_{G}))=n+c$, where $<$ denotes the lexicographic order induced by $x_{1}>\cdots>x_{n}>y_{1}>\cdots>y_{n}$ in the ring $S=K[x_{1},\ldots,x_{n},y_{1},\ldots,y_{n}]$.

We begin with the following lemma.

\begin{Lemma}\label{2.1}
Let $G$ be a graph on the vertex set $[n]$ and let $i\in[n]$. Then $$\ini_{<}(J_{G},x_{i},y_{i})=(\ini_{<}(J_{G}),x_{i},y_{i}).$$
\end{Lemma}

\begin{proof}
We have $\ini_{<}(J_{G},x_{i},y_{i})=\ini_{<}(J_{G\setminus \{i\}},x_{i},y_{i})=(\ini_{<}(J_{G\setminus \{i\}}),x_{i},y_{i})$. Therefore, we have to show that $(\ini_{<}(J_{G}),x_{i},y_{i})=(\ini_{<}(J_{G\setminus \{i\}}),x_{i},y_{i})$. The inclusion $\supseteq$ is obvious since $J_{G\setminus \{i\}}\subset J_{G}$. For the other inclusion, let us take $u$ to be a minimal generator of $\ini_{<}(J_{G})$. If $x_{i}\mid u$ or $y_{i}\mid u$, obviously $u\in (\ini(J_{G\setminus\{i\}}),x_{i},y_{i})$. Let now $x_{i}\nmid u$ and $y_{i}\nmid u$. This means that $u=u_{\pi}x_{k}y_{l}$ for some admissible path $\pi$ from $k$ to $l$ which does not contain the vertex $i$. Then it follows that $\pi$ is a path from $k$ to $l$ in $G\setminus \{i\}$, hence $u\in \ini_{<}(J_{G\setminus \{i\}})$.
\end{proof}

\begin{Theorem}\label{2.2}
Let $G$ be a block graph. Then $\depth(S/J_{G})=\depth(S/\ini_{<}(J_{G}))=n+c$, where $c$ is the number of connected component of $G$.
\end{Theorem}

\begin{proof}
Let $G_{1},\ldots,G_{c}$ be the connected components of $G$ and $S_{i}=K[\{x_{j},y_{j}\}_{j\in G_{i}}]$. Then $S/J_{G}\cong S_{1}/J_{G_{1}}\otimes\cdots\otimes S_{c}/J_{G_{c}}$, so that $\depth S/J_{G}=\depth S_{1}/J_{G_{1}}+\cdots+\depth S_{c}/J_{G_{c}}$. Moreover, we have  $S/\ini_{<}(J_{G})\cong S/\ini_{<}(J_{G_{1}})\otimes\cdots\otimes S/\ini_{<}(J_{G_{c}})$, thus $\depth S/\ini_{<}(J_{G})\cong\depth S_{1}/\ini_{<}(J_{G_{1}})+\cdots+\depth S_{c}/\ini_{<}(J_{G_{c}})$.
Therefore, without loss of generality, we may assume that $G$ is connected. By \cite[Theorem 1.1]{EHH} we know that $\depth(S
/J_{G})=n+1$. In order to show that $\depth(S/\ini_{<}(J_{G}))=n+1$, we proceed by induction on the number of maximal cliques
of $G$. Let $\Delta(G)$ be the clique complex of $G$ and let $F_{1},\ldots,F_{r}$ be a leaf order on the facets of $\Delta(G)
$. If $r=1$, then $G$ is a simplex and the statement is well known. Let $r>1$; since $F_{r}$ is a leaf, there exists a
unique vertex, say $i\in F_{r}$, such that $F_{r}\cap F_{j}=\{i\} $ where $F_{j}$ is a branch of $F_{r}$. Let $F_{t_{1}},
\ldots,F_{t_{q}}$ be the facets of $\Delta(G)$ which intersect the leaf $F_{r}$ in the vertex $\{i\}$. Following the proof
of \cite[Theorem 1.1]{EHH} we may write $J_{G}=J_{1}\cap J_{2}$ where $J_{1}=\bigcap_{i\notin \mathcal{S}} P_{\mathcal{S}}(G)
$ and $J_{2}= \bigcap_{i\in \mathcal{S}} P_{\mathcal{S}}(G)$. Then, as it was shown in the proof of \cite[Theorem 1.1]{EHH},
it follows that $J_{1}=J_{G'}$ where $G'$ is obtained from $G$ by replacing the cliques $F_{t_{1}},\ldots,F_{t_{q}}$ and $F_{
r}$ by the clique on the vertex set $F_{r}\cup(\bigcup_{j=1}^q F_{t_{j}})$. Also, $J_{2}=(x_{i},y_{i})+J_{G''}$ where $
G''$ is the restriction  of $G$ to the vertex set $[n]\setminus \{i\}$.

We have $\ini_{<}(J_{G})=\ini_{<}(J_{1}\cap J_{2})$. By \cite[Lemma 1.3]{Concoa}, we have $\ini_{<}(J_{1}\cap J_{2})=\ini_{<}
(J_{1})\cap\ini_{<}(J_{2})$ if and only if $\ini_{<}(J_{1}+J_{2})=\ini_{<}(J_{1})+\ini_{<}(J_{2})$. But  $\ini_{<}(J_{1}+J_{2
})=$$\ini_{<}(J_{G'}+(x_{i},y_{i})+J_{G''})=$$\ini_{<}(J_{G'}+(x_{i},y_{i}))$. Hence, by Lemma ~\ref{2.1}, we get $\ini_{<}(J
_{1}+J_{2})=$$\ini_{<}(J_{G'})+(x_{i},y_{i})=\ini_{<}(J_{1})+\ini_{<}(J_{2})$. Therefore, we get $\ini_{<}(J_{G})=\ini_{<}(J_
{1})\cap\ini_{<}( J_{2})$ and, consequently, we have the following exact sequence of $S$-modules
$$ 0 \longrightarrow \frac{S}{\ini_{<}(J_{G})} \longrightarrow \frac{S}{\ini_{<}(J_{1})}\oplus\frac{S}{\ini_{<}(J_{2})}
\longrightarrow \frac{S}{\ini_{<}(J_{1}+ J_{2})}\longrightarrow 0.$$
By using again Lemma~\ref{2.1}, we have $\ini_{<}(J_{2})=\ini_{<}((x_{i},y_{i}),J_{G''})=(x_{i},y_{i})+\ini_{<}(J_{G''}).$
 Thus, we have actually the following exact sequence
\begin{equation}\label{1}
    0 \longrightarrow \frac{S}{\ini_{<}(J_{G})} \longrightarrow \frac{S}{\ini_{<}(J_{G'})}\oplus\frac{S}{(x_{i},y_{i})+\ini_{
		<}(J_{G''})}\longrightarrow \frac{S}{(x_{i},y_{i})+ \ini_{<}(J_{G'})}\longrightarrow 0.
\end{equation}

Since $G'$ inherits the properties of $G$ and has a smaller number of maximal cliques than $G$, it follows, by the inductive
hypothesis, that $\depth(S/J_{G'})=\depth(S/\ini_{<}(J_{G'}))=n+1$. Let $S_{i}$ be the polynomial ring $S/(x_{i},y_{i})$.
Then $S/((x_{i},y_{i})+\ini_{<}(J_{G''}))\cong S_{i}/\ini_{<}(J_{G''})$. Since $G''$ is a graph on $n-1$ vertices with $q+1$
connected components and satisfies our conditions, the inductive hypothesis implies that $\depth S/((x_{i},y_{i})+\ini_{<}(J_
{G''}))=n+q\geq n+1$. Hence,
\[\depth (S/\ini_{<}(J_{G'})\oplus S/((x_{i},y_{i})+\ini_{<}(J_{G''})))=n+1.\]
Next, we observe that $S/((x_{i},y_{i})+\ini_{<}(J_{G'}))\cong S_{i}/\ini_{<}(J_{H})$, where $H$ is obtained from $G'$ by
replacing the clique on the vertex set $F_{r}\cup(\bigcup_{j=1}^q F_{t_{j}})$ by the clique on the vertex set $F_{r}\cup(
\bigcup_{j=1}^q F_{t_{j}})\setminus \{i\}$. Hence, by the inductive hypothesis, $\depth(S/((x_{i},y_{i})+\ini_{<}(J_{G'})))=n
$ since $H$ is connected and its vertex set has cardinality $n-1$. Hence, by applying the Depth lemma to our exact sequence$~
(\ref{1})$, we get $\depth S/J_{G}=\depth S/\ini_{<}(J_{G})=n+1$.
\end{proof}

Let $G$ be a connected graph on the vertex set $[n]$ which consists of
\begin{itemize}
	\item [(i)] a sequence of maximal cliques $F_1,\ldots,F_\ell$ with $\dim F_i\geq 1$ for all $i$ such that $|F_i\cap F_{i+1}|=1$ for $1\leq i\leq \ell-1$ and $F_i\cap F_j=\emptyset$ for any
	$i<j$ such that $j\neq i+1,$ together with
	\item [(ii)] some additional edges of the form $F=\{j,k\}$ where $j$  is an intersection point of two consecutive cliques $F_i,F_{i+1}$ for some $1\leq i\leq \ell-1$, and $k$ is a vertex of degree $1.$
\end{itemize}

In other words, $G$ is obtained from a graph $H$ with $\Delta(H)=\langle F_{1},\ldots,F_{l}\rangle$ whose binomial edge ideal is Cohen-Macaulay (see \cite[Theorem 3.1]{EHH}) by attaching edges in the intersection points of the facets of $\Delta(H).$
Therefore, $G$ looks like the graph displayed in Figure~\ref{example}.

\begin{figure}[hbt]
\begin{center}
\psset{unit=1cm}
\begin{pspicture}(1,-2)(8,3)
\pspolygon(0,0)(1,-1)(2,0)(1,1)
\psline(2,0)(4,0)
\psline(2,0)(3,1)
\pspolygon(4,0)(5,-1)(6,0)(5,1)
\pspolygon(6,0)(7.5,0)(7.5,1)(6,1)
\pspolygon(7.5,0)(8.5,-1)(8.5,1)

\psline(2,0)(1,-1)
\psline(2,0)(2,-1)
\psline(4,0)(4,1)
\psline(6,0)(6,-1)
\psline(6,0)(6.8,-1)
\psline(0,0)(2,0)
\psline(1,1)(1,-1)
\psline(4,0)(6,0)
\psline(5,1)(5,-1)
\psline(6,0)(7.5,1)
\psline(6,1)(7.5,0)
\psline(2,0)(3,-1)
\rput(0,0){$\bullet$}
\rput(1,1){$\bullet$}
\rput(1,-1){$\bullet$}
\rput(2,0){$\bullet$}
\rput(2,-1){$\bullet$}
\rput(3,-1){$\bullet$}
\rput(4,1){$\bullet$}
\rput(3,1){$\bullet$}
\rput(4,0){$\bullet$}
\rput(5,-1){$\bullet$}
\rput(6,0){$\bullet$}
\rput(5,1){$\bullet$}
\rput(6,-1){$\bullet$}
\rput(6.8,-1){$\bullet$}
\rput(6,1){$\bullet$}
\rput(7.5,0){$\bullet$}
\rput(7.5,1){$\bullet$}
\rput(8.5,1){$\bullet$}
\rput(8.5,-1){$\bullet$}

\end{pspicture}
\end{center}
\caption{$C_{\ell}$-graph}
\label{example}
\end{figure}
Such a graph  has, obviously, the property that its longest induced path has length equal to $\ell.$ If a connected graph $G$ satisfies the above conditions (i) and (ii), we say that $G$ is  a $\Cc_\ell$-graph. In the case that $\dim F_i=1$ for  $1\leq i\leq \ell$, then $G$ is called a {\em caterpillar graph}.

We should also note that any $\Cc_\ell$--graph is  chordal  and has the property that any two distinct maximal cliques intersect in at most one vertex. So that any $C_{\ell}$-graph is a connected block graph.

\begin{Theorem}\label{clgraph}
Let $G$ be a $\Cc_\ell$-graph on the vertex set $[n].$ Then
 \[\reg(S/J_G)=\reg(S/\ini_< (J_G))=\ell. \]
\end{Theorem}

\begin{proof}
Let $G$ consists of the sequence of maximal cliques $F_{1},\ldots,F_{\ell}$ as in condition (i) to which we add  some edges as in condition (ii). So the maximal cliques of $G$ are $F_{1},\ldots,F_{\ell}$ and all the additional whiskers. We proceed by induction on the number $r$ of maximal cliques of $G$. If $r=\ell$, then $G$ is a closed graph whose binomial edge ideal is Cohen-Macaulay, hence the statement holds by \cite[Theorem 3.2]{EZ}. Let $r>\ell$ and let $F'_{1},\ldots,F'_{r}$ be a leaf order on the facets of $\Delta(G)$. Obviously, we may choose a leaf order on $\Delta (G)$ such that $F'_{r}=F_{\ell}$. With the same arguments and notation as in the proof of Theorem~\ref{2.2}, we get the sequence $(\ref{1})$.

We now observe that $G'$ is a $\Cc_{\ell-1}$-graph, hence, by the inductive hypothesis,
\begin{equation}\label{2}
\reg\frac{S}{J_{G'}}=\reg\frac{S}{\ini_{<}(J_{G'})}=\ell-1.
\end{equation}

The graph $G''$ has at most two non-trivial connected components. One of them, say $H_{1}$, is a $\Cc_{\ell'}$-graph with $\ell'\in\{\ell-2,\ell-1\}$. The other possible non-trivial component, say $H_{2}$, occurs if $|F_{\ell}|\geq 3$ and, in this case, $H_{2}$ is a clique of dimension $|F_{\ell}|-2\geq 1$. By the inductive hypothesis, we obtain
\begin{equation}\label{3}
\reg\frac{S}{J_{G''}}=\reg\frac{S}{\ini_{<}(J_{G''})}=\reg\frac{S}{J_{H_{1}}}+\reg\frac{S}{J_{H_{2}}}\leq \ell-1+1=\ell.
\end{equation}
Relations $(\ref{2})$ and $(\ref{3})$ yield $\reg (S/\ini_{<}(J_{G'})\oplus S/((x_{i},y_{i})+\ini_{<}(J_{G''})))\leq \ell$. From the exact sequence $(\ref{1})$ we get
\begin{equation}\label{4}
\reg\left(\frac{S}{\ini_{<}(J_{G})}\right)\leq \max\{\reg\left(\frac{S}{\ini_{<}(J_{G'})}\oplus\frac{S}{(x_{i},y_{i})+\ini_{<}(J_{G''})}\right), \reg\frac{S}{\ini_{<}(J_{G'})}+1\}\leq \ell.
\end{equation}

By \cite[Theorem 3.3.4]{HH10}, we know that  $\reg (S/J_{G})\leq \reg(S/\ini_{<}(J_{G})$, and by \cite[Theorem 1.1]{MM}, we have  $\reg(S/J_{G})\geq \ell$. By using all these inequalities, we get the desired conclusion.
\end{proof}

\section{Binomial edge ideals of caterpillar trees}
\label{tree}


Matsuda and Murai showed in \cite{MM} that, for any connected graph $G$ on the vertex set $[n]$, we have $\ell \leq \reg(S/J_{G}) \leq n-1$, where $\ell$ denotes the length of the longest induced path of $G$, and conjectured that $\reg (S/J_{G})=n-1$ if and only if $T$ is a line graph. Several recent papers are concerned with this conjecture; see, for example, \cite{EZ}, \cite{Sara}, and \cite{Sara1}. One may ask as well to characterize connected graphs $G$ whose longest induced path has length $\ell$ and $\reg(S/J_{G})=\ell$. In this section, we answer this question for trees.

A caterpillar tree is a tree $T$ with the property that it contains a path $P$ such that any vertex of $T$ is either a vertex of $P$ or it is adjacent to a vertex of $P$. Clearly, any caterpillar tree is a $\Cc_{\ell}$-graph for some positive integer $\ell$.
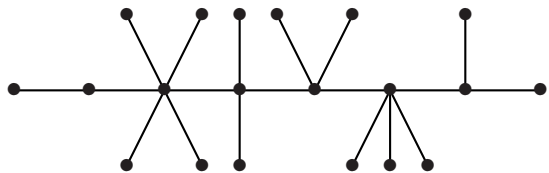
\begin{figure}[hbt]
\begin{center}
\psset{unit=1cm}
\begin{pspicture}(1,-2)(6.25,2)

\psline(0,0)(1,0)
\psline(1,0)(2,0)
\psline(2,0)(3,0)
\psline(3,0)(4,0)
\psline(4,0)(5,0)
\psline(6,0)(5,0)
\psline(6,0)(7,0)
\psline(2,0)(1.5,-1)
\psline(2,0)(1.5,1)
\psline(2,0)(2.5,-1)
\psline(2,0)(2.5,1)
\psline(3,0)(3,1)
\psline(3,0)(3,-1)
\psline(3.5,1)(4,0)
\psline(4,0)(4.5,1)
\psline(5,0)(4.5,-1)
\psline(5,0)(5,-1)
\psline(5,0)(5.5,-1)
\psline(6,0)(6,1)
\rput(0,0){$\bullet$}
\rput(1,0){$\bullet$}
\rput(3,0){$\bullet$}
\rput(2,0){$\bullet$}
\rput(4,0){$\bullet$}
\rput(5,0){$\bullet$}
\rput(6,0){$\bullet$}
\rput(7,0){$\bullet$}
\rput(1.5,1){$\bullet$}
\rput(2.5,1){$\bullet$}
\rput(1.5,-1){$\bullet$}
\rput(2.5,-1){$\bullet$}
\rput(3,-1){$\bullet$}
\rput(3,1){$\bullet$}
\rput(3.5,1){$\bullet$}
\rput(4.5,1){$\bullet$}
\rput(4.5,-1){$\bullet$}
\rput(5.5,-1){$\bullet$}
\rput(6,1){$\bullet$}
\rput(5,-1){$\bullet$}

\end{pspicture}
\end{center}
\caption{Caterpillar}
\label{example2}
\end{figure}

\begin{figure}[hbt]
\begin{center}
\psset{unit=1cm}
\begin{pspicture}(1,-2)(6.25,2)

\psline(0,0)(1,0)
\psline(1,0)(2,0)
\psline(2,0)(3,0)
\psline(3,0)(4,0)
\psline(4,0)(5,0)
\psline(6,0)(5,0)
\psline(6,0)(7,0)
\psline(4,0)(4.5,-1)
\psline(4.5,-1)(5,-2)
\rput(0,0){$\bullet$}
\rput(1,0){$\bullet$}
\rput(3,0){$\bullet$}
\rput(2,0){$\bullet$}
\rput(4,0){$\bullet$}
\rput(5,0){$\bullet$}
\rput(6,0){$\bullet$}
\rput(7,0){$\bullet$}
\rput(4.5,-1){$\bullet$}
\rput(5,-2){$\bullet$}

\end{pspicture}
\end{center}
\caption{Induced graph H}
\label{example3}
\end{figure}
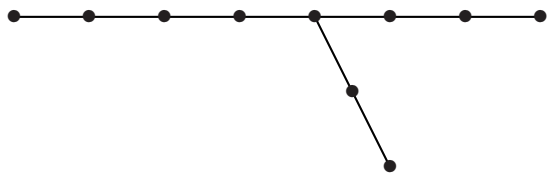
Caterpillar trees were first studied by Harary and Schwenk \cite{HF}. These graphs have applications in chemistry and physics \cite{EB}. In Figure~\ref{example2}, an example of caterpillar tree is displayed. Note that any caterpillar tree is a narrow graph in the sense of Cox and Erskine \cite{Cox}. Conversely, one may easily see that any narrow tree is a caterpillar tree. Moreover, as it was observed in \cite{M}, a tree is a caterpillar graph if and only if it is weakly closed in the sense of definition given in \cite{M}.

In the next theorem we characterize the trees $T$ with $\reg(S/J_{T})=\ell$ where $\ell$ is the length of the longest induced path of $T$.

\begin{Theorem}\label{3.1}
Let $T$ be a tree on the vertex set $[n]$ whose longest induced path $P$ has length $\ell$. Then $\reg(S/J_{T})=\ell$ if and only if $T$ is caterpillar.
\end{Theorem}
\begin{proof}
Let $T$ be a caterpillar tree whose longest induced path has length $\ell$. Then, by the definition of a caterpillar tree, it follows that $T$ is a $\Cc_{\ell}$-graph. Hence, $\reg(S/J_{T})=\ell$ by Theorem~\ref{clgraph}. Conversely, let $\reg(S/J_{T})=\ell$ and assume that $T$ is not caterpillar. Then $T$ contains an induced subgraph $H$ with $\ell+3$ vertices as in Figure~\ref{example3}.

Then, by \cite[Theorem 27]{SZ}, it follows that $\reg(S/J_{H})=\ell+1$. Thus, since $\reg(S/J_{H})\leq \reg(S/J_{G})$  (see \cite[Corollary 2.2]{MM}), it follows that $\reg (S/J_{G})\geq \ell+1$, contradiction to our hypothesis.
\end{proof}

{}

\end{document}